\documentclass[11pt,english]{amsart}
\usepackage{amsmath,amssymb,amsthm}
\usepackage{hyperref}
\usepackage{float}
\usepackage{tikz}
\usepackage{graphicx}
\usetikzlibrary{calc}
\usetikzlibrary{arrows,decorations.pathmorphing,backgrounds,positioning,fit,petri}

\newtheorem{theorem}{Theorem}
\newtheorem{lemma}[theorem]{Lemma}
\newtheorem{corollary}[theorem]{Corollary}

\newtheorem{proposition}[theorem]{Proposition}
\theoremstyle{definition}
\newtheorem{definition}[theorem]{Definition}
\newtheorem{remark}[theorem]{Remark}
\newtheorem{example}[theorem]{Example}

\newcommand{\Aut}{\hbox{\rm Aut}}

\newcommand{\Sym}{{\rm Sym}}
\newcommand{\inv}{^{-1}}

\newcommand{\lin}{\overline}

\newcommand{\Cos}{\hbox{\rm Cos}}
\newcommand{\Cov}{\hbox{\rm Cov}}
\newcommand{\core}{\hbox{\rm core}}

\newcommand{\GC}{\hbox{\rm GenCov}}

\newcommand{\fib}{{\rm{fib}}}

\newcommand{\invG}{\mathop{{\rm inv}_\Gamma}}
\newcommand{\begP}{\mathop{{\rm beg}'}}
\newcommand{\invI}{\mathop{{\rm inv}_I}}
\newcommand{\invP}{\mathop{{\rm inv}'}}
\renewcommand{\inv}{\mathop{{\rm inv}}}
\DeclareMathOperator{\Inv}{{\rm inv}}
\DeclareMathOperator{\beg}{{\rm beg}}
\DeclareMathOperator{\term}{{\rm end}}
\newcommand{\id}{\hbox{id}}

\newcommand{\D}{{\rm D}}
\newcommand{\V}{{\rm V}}

\newcommand{\bh}{{\bar{h}}}

\makeatletter
\numberwithin{equation}{section}
\numberwithin{figure}{section}

\makeatother

\usepackage{babel}

\tikzstyle{vertu}=[circle,draw=black!,fill=black!,thick]
\tikzstyle{vertv}=[circle,draw=black!,fill=white!,thick]
\tikzstyle{blank}=[circle,draw=white!,fill=white!,thick,inner sep = 0.11mm]
\tikzstyle{post}=[->,shorten >=1pt,>=latex,thick]
\tikzstyle{dentro}=[<-,>=latex,thick]

\title[]{Generalised voltage graphs}
\author{Primo\v{z} Poto\v{c}nik}
\author{Micael Toledo}

\address{Primo\v{z} Poto\v{c}nik, Faculty of Mathematics and Physics, University of Ljubljana, Jadranska 21, SI-1000 Ljubljana, Slovenia.\newline
\indent Also affiliated with: Institute of Mathematics, Physics and Mechanics, Jadranska 19, SI-1000 Ljubljana, Slovenia.
}
\email{primoz.potocnik@fmf.uni-lj.si}

\address{Micael Toledo, Institute of Mathematics, Physics and Mechanics, Jadranska 19, SI-1000.\newline
 Also affiliated with: University of Primorska, Faculty of Mathematics, Natural Sciences and Information Technologies, Glagolja\v{s}ka 8, SI-6000 Koper, Slovenia.}
\email{micael.toledo@imfm.uni-lj.si}

\thanks{The authors gratefully acknowledge support of the Slovenian Research Agency: Core Programme P1-0294, Research Project J1-1691 and the Young Researcher Scholarship programme.}

\begin{document}

\begin{abstract}A graph with a semiregular group of automorphisms can be thought of as the derived cover arising from a voltage graph. Since its inception, the theory of voltage graphs and their derived covers has been a powerful tool used in the study of graphs with a significant degree of symmetry. We generalise this theory to graphs with a group of automorphisms that is not necessarily semiregular, and we generalise several well-known results of the classical theory of voltage graphs to this broader setting.  
\end{abstract}
\maketitle

\section{Introduction}
\label{sec:intro}

Graphs (and mathematical objects in general) that posses non-trivial symmetry have many nice features, one of them being that they allow a more compact description, which not only saves space for storage but also enables more efficient analysis of the graph. Let us mention two typical examples of this phenomenon. 

Suppose that $\Gamma$ is a connected graph and $G$ a 
group of automorphisms of $\Gamma$.
If $G$ acts transitively on the arcs (ordered pairs of adjacent vertices) of $\Gamma$, then
$\Gamma$ can be reconstructed as a coset-graph from the group $G$, the vertex-stabiliser $G_v$ and an element $a\in G$ where $a$ is an element swapping the vertex $v$ with a neighbour of $v$;
see \cite[Section 1.2]{ConderRogla} or \cite{cosets}, for example.
Describing arc-transitive graphs as coset graphs
is a standard method often used in the classification results (see \cite{Cos1,Cos2,Cos3,Cos4,Cos5}, to name a few) as well as a way to store a graph in a database (see \cite{census}).

At the other extreme, if a group of automorphisms $G$ acts semiregularly on the vertices of $\Gamma$ (that is, $G_v = 1$ for every vertex $v\in \V(\Gamma)$) one can reconstruct $\Gamma$ as the derived covering graph $\Cov(\Gamma/G,\zeta)$, where  $\Gamma/G$ is the quotient graph and $\zeta \colon \D(\Gamma) \to G$  is a mapping called a voltage assignment. The theory of graph covers and their description in terms of
voltages has a long history going back to the work of Gross and Tucker \cite{gross,grosstuc}
and has now become one of the central tools in the theory of symmetries of graphs (see \cite{CovApp1,volt1,CovApp3,volt2, lift,elabcov,CovApp2}, for exmaple). 

In the case when $G$ is an arbitrary group of automorphisms of $\Gamma$
(not necessarily  semiregular or arc-transitive)
no similar method which encodes a complete information about $\Gamma$
has been described in the literature so far
even though the usefulness of such a potential method has been discussed
on several occasions (for example in \cite{macaj}).

The aim of this paper is to present such a method, which generalises both
the coset graph construction as well as the covering graph construction based
on the voltage assignments.

In what follows, we first introduce all the necessary notation needed
to present this construction and then state the main results of the paper. We begin with
 a definition of a {\em graph} as introduced in \cite{lift}, which proves to be the most suitable for our purposes (more terminology pertaining to this definition of a graph 
 will be given in Section~\ref{sec:def}).

A {\em graph} is a quadruple $(V,D,\beg,\inv)$ where $V$ is a non-empty set of {\em vertices}, $D$ is a set of {\em darts} (also known as {\em arcs} in the literature), $\beg\colon D \to V$ is a function assigning to each dart $x$ its {\em initial vertex} $\beg x$, and
$\inv\colon D \to D$ is a function assigning to each dart $x$ its {\em inverse dart} $\inv x$ (also denoted $x^{-1}$ when there is no danger for  ambiguity) satisfying 
$\inv\inv x =x$ 
 for every dart $x$.
 If $\Gamma = (V,D,\beg,\inv)$ is a graph, then we let
$\V(\Gamma):=V$, $\D(\Gamma) := D$, $\beg_\Gamma:=\beg$ and $\inv_\Gamma:=\inv$.

For a group $G$, we let $S(G)$ denote the set of all subgroups of $G$ and for $g,h\in G$, we let $g^h:=h^{-1}gh$ be the conjugate of $g$ by $h$.

 We can now present the construction which generalises that of the derived covering graph
 introduced in the classical work of Gross and Tucker  \cite[Section 2.1.1]{gross}.

\begin{definition}
Let $\Delta$ be a connected graph, let $G$ be a group, and 
let $\omega\colon \V(\Delta) \cup \D(\Delta) \to S(G)$ and $\zeta\colon \D(\Delta) \to G$ be two functions such that the following hold for all $x \in \D(\Delta)$:
\begin{align}
\label{eq:volt1} \omega(x) \leq  \omega(\beg_\Delta x);\\ 
\label{eq:volt2} \omega(x) = \omega(x^{-1})^{\zeta(x)};\\
\label{eq:volt3} \zeta(x^{-1})\zeta(x) \in \omega(x). 	
\end{align}
We then say that the quadruple $(\Delta,G,\omega,\zeta)$ is a {\em generalised voltage graph} and we call the functions $\omega$ and $\zeta$ a {\em weight function} and a {\em voltage assignment}, respectively. 
\end{definition}

\begin{definition}
\label{def:gencov}
Let $(\Delta,G,\omega,\zeta)$ be a generalised voltage graph and let $\Gamma$ be the graph defined by:
\begin{itemize}
\item $\V(\Gamma) = \{(v,\omega(v)g) \mid g \in G, v \in \V(\Delta)\}$;
\item $\D(\Gamma) = \{(x,\omega(x)g) \mid g \in G, x \in \D(\Delta)\}$;
\item $\beg_{\Gamma}(x, \omega(x)g)=(\beg_{\Delta} (x), \omega(\beg_{\Delta} x)g)$;
\item $\inv_{\Gamma}(x, \omega(x)g)=(\inv_\Delta x, \omega(\inv_\Delta x)\zeta(x) g)$.
\end{itemize}
Then $\Gamma$ is called the {\em generalised cover} arising from $(\Delta,G,\omega,\zeta)$ and is denoted by $\GC(\Delta,G,\omega,\zeta)$.
\end{definition}

One should of course check that the functions $\beg_\Gamma$ and $\inv_\Gamma$ from Definition~\ref{def:gencov} are well defined and that $\inv_\Gamma$ is indeed an involution on $\D(\Gamma)$.
We do that in Lemma~\ref{lem:OK} in Section~\ref{sec:def}. Two explicit examples illustrating this construction are presented in Example~\ref{ex:S6}.

In order to formulate a result (Theorem~\ref{the:lift} below) which can be seen as the main motivation for the introduction of the generalised voltage graphs, we need to introduce the concept  a quotient graph arising from an action of a group of automorphisms.
For a group $G$ acting on a set $\Omega$ we let $\Omega/G$ denote the
set of orbits of this action; that is $\Omega/G = \{x^G : x \in \Omega\}$.
A subset of $\Omega$ which contains precisely one element from each orbit in $\Omega/G$ is called a {\em transversal} of $\Omega/G$.

\begin{definition}
Let $\Gamma=(V,D,\beg,\Inv)$ be a graph and let $G \leq \Aut(\Gamma)$. 
Let $\beg' \colon D/G \to V/G$ and $\inv'\colon D/G \to D/G$ be mappings defined by
$\beg'(d^{G})= (\beg d)^{G}$ and $\Inv'd^G= (\Inv d)^G$. Then we call the
graph $(V/G,D/G,\beg',\Inv')$  the {\em $G$-quotient} of $\Gamma$ and denote it by $\Gamma/G$. The mapping $f:\V(\Gamma) \cup \D(\Gamma) \to \V(\Gamma/G) \cup \D(\Gamma/G)$ given by $f(x) = x^G$, is called the {\em quotient map} associated with $\Gamma$ and $G$. 
\end{definition}

\begin{theorem}
\label{the:lift}
Let $\Gamma$ be a graph and let $G \leq \Aut(\Gamma)$. Then there exist functions $\zeta \colon \D(\Gamma/G) \to G$ and $\omega \colon \V(\Gamma/G) \cup \D(\Gamma/G) \to S(G)$, such that $(\Gamma/G,G,\omega,\zeta)$ is a generalised voltage graph and $\Gamma$ is isomorphic to the associated generalised covering graph $\GC(\Gamma/G,G,\omega,\zeta)$.
\end{theorem}

The above theorem is proved in Section~\ref{sec:reconstruction}. In fact,
there we state and prove a more detailed version (see Thereom~\ref{the:lift2}), where
more information is given on the weight function $\omega$ and the generalised voltage assignment
$\zeta$.

Let us now consider three special cases of generalised covers. If $G$ acts semiregularly on $\V(\Gamma)$ then the weights $\omega(x)$ for $x \in \V(\Gamma) \cup \D(\Gamma)$ appearing in Theorem \ref{the:lift} are all trivial (as explicitly stated in Theorem \ref{the:lift2}). Consequently the generalised voltage assignment $\zeta$ satisfies the condition $\zeta(x^{-1})=\zeta(x)^{-1}$ for all $x \in \D(\Gamma)$ and can thus be viewed as a voltage assignment as defined in \cite{lift}. The generalised covering graph $\GC(\Gamma,G,\omega,\zeta)$ then coincides with the derived graph $(\Gamma /G)^{\zeta}$ as defined in \cite[Section 2.1.1]{gross} and Theorem \ref{the:lift} can be viewed as a generalisation of the classical result \cite[Theorem 2.2.2]{gross} of Gross and Tucker.
 
The second special case which we want to point out is when $G$ acts transitively
on the arcs and vertices of $\Gamma$. In this case the quotient $\Gamma/G$ is a vertex with a single semi-edge attached to it. Let $v$ be an arbitrary vertex of $\Gamma$ and let $x$ be a dart emanating from $v$. As can be deduced from Theorem \ref{the:lift2} the weight function $\omega$ appearing in Theorem \ref{the:lift} is given by $\omega(v^G)=G_v$ and $\omega(x)=G_x$. Furthermore $\zeta(x^G)$ is an element $a\in G$ swapping $x$ with its inverse $x^{-1}$. As one
can easily see, the generalised covering construction with $(\Gamma/G,G,\omega,\zeta)$ is then essentially the same as the coset graph construction (see \cite[Section 1.2]{ConderRogla}, for example).

Finally, suppose that $G$ acts locally arc-transitively (that is, for every vertex $v$ the stabiliser $G_v$ acts transitively on the arcs emanating from $v$) but intransitively on the vertices of $\Gamma$.
For the sake of simplicity we also assume that $\Gamma$ has no isolated vertices (that is, that every vertex has an arc emanating from it).
These graphs have been extensively studied (see \cite{locallyAT} for a nice overview of the topic).
One can easily see that then $G$ acts transitively on the edges and that it has precisely two orbits on the vertices. In particular,  the quotient $\Gamma/G$ is isomorphic to $K_2$ (the graph with two vertices and a single edge connecting them). Let $u$ and $v$ be two adjacent vertices of $\Gamma$
and let $x$ be a dart with $\beg x = v$ and $\term x = u$. Then, in view of
Theorem~\ref{the:lift2}, the function $\omega$  featuring in Theorem~\ref{the:lift} satisfies $\omega(v^G) = G_v$, $\omega(u^G) = G_u$
and $\omega(x^G) = \omega((x^{-1})^G) = G_{u}\cap G_v$. Furthermore the function
$\zeta$ can be chosen to be trivial on both darts of $\Gamma/G$. The generalised covering construction applied to such a data is then essentially the same as the
construction now known as the bicoset construction (see \cite[Definition 2.1]{DuXu}, or see \cite{locallyAT}, where
this construction is called the {\em coset graph construction}, or see \cite[p.\ 380]{Gold80} for a historical reference, where
the resulting graph was called the {\em graph of the completion} of an amalgam).

When analysing the connectivity properties of the generalised covers (see Section~\ref{sec:con}), we obtained, as a byproduct,
Theorem \ref{the:genGv}, which also generalises some well-known facts from algebraic graph theory. We do not state Theorem \ref{the:genGv} here, as we have not yet set the necessary terminology to do so. Instead, we state a simpler but interesting special case.

\begin{proposition}
\label{prop:con}
Let $\Gamma$ be a graph, let $G \leq \Aut(\Gamma)$ be such that $\Gamma / G$ is a tree and let $\Gamma'$ be a subgraph of $\Gamma$ that is mapped isomorphically to $T$ by the quotient map. Then, $\Gamma$ is connected if and only if $G = \langle   G_v \mid v \in \V(\Gamma') \rangle$.
\end{proposition}

A well-known
 special case of the above proposition occurs when
 the quotient $\Gamma/G$ is isomorphic
to $K_2$. As we described above, $\Gamma$ is then isomorphic to the bicoset graph arising from
$G$ and the stabilisers $G_v$ and $G_u$ of two adjacent vertices.
Proposition~\ref{prop:con} then claims that
 $\Gamma$ is connected if and only if
$G = \langle G_v, G_u\rangle$ (see \cite[Lemma 3.7]{locallyAT}).

Another well-known special case of Theorem~\ref{the:genGv} occurs when $G$ acts transitively
on the arcs of $\Gamma$. As mentioned above, the quotient is a vertex with a single semi-edge attached and the generalised cover isomorphic to $\Gamma$ is essentially the coset
graph $\Cos(G_v,a)$ where $a$ is an element of $G$ inverting an arc emanating from a vertex $v$.
Theorem~\ref{the:genGv} then asserts that $\Gamma$ is connected if and only if
$G=\langle G_v, a\rangle$ (see \cite[Lemma 2.1]{ConderRogla}).

In Section~\ref{sec:def},
we present some further definitions. Section~\ref{sec:reconstruction} is devoted to the
proof of Theorem~\ref{the:lift} In Section~\ref{sec:auto} we investigate the action of $G$ on the generalised cover $\GC(\Delta,G,\omega,\zeta)$ and we show some natural isomorphisms between generalised covers. In Section~\ref{sec:normal} we give a generalisation of a well-known result stating that every voltage assignment can be thought of as being trivial on a prescribed spanning tree of a voltage graph. Section~\ref{sec:con} is devoted to a characterisation of connectivity of a generalised cover and we prove a more general version of Proposition \ref{prop:con}. In Section~\ref{sec:simp} we give necessary and sufficient conditions for a generalised cover to be a simple graph.

\section{Further definitions and basic results}
\label{sec:def}

In this section we give some further definitions regarding the generalised voltage graphs and generalised covers, provide some examples and prove some basic results. 

\subsection{Graphs}

We begin by setting  terminology regarding graphs, as defined in Section~\ref{sec:intro}.
Let $\Gamma=(V,D,\beg,\inv)$ be a graph. For $x\in D$, we call the vertex $\beg x^{-1}$ 
the {\em end} of $x$ and denote it by $\term x$.
 Two darts $x$ and $y$ are {\em parallel} if $\beg x = \beg y$ and $\term x = \term y$. 
An {\em edge} of a graph $\Gamma$ is a pair $\{x,x^{-1}\}$ where $x$ is a dart of $\Gamma$.
The vertices $\beg x$ and $\beg x^{-1}$ are then called the {\em endvertices} of the edge. 
An edge $\{x,x^{-1}\}$ is a {\em semi-edge} if $x=x^{-1}$ and is a {\em loop} if $x \neq x^{-1}$ but $\beg x = \term x$.

Two distinct edges $e$ and $e'$ are said to be {\em parallel} if there is a dart in $e$ that is parallel to a dart in $e'$, or equivalently, if $e$ and $e'$ have the same endvertices.
 A graph that has no semi-edges, no loops and no pairs of distinct parallel edges
is {\em simple}. Note that a simple graph is uniquely determined by its vertex-set and its edge-set and the usual terminology of simple graphs applies. In particular, a dart in a simple graph is usually called an {\em arc}.

The {\em neighbourhood}  of a vertex $v\in \V(\Gamma)$ is defined as the set $\Gamma(v):=\{x \in \D(\Gamma) : \beg x = v\}$ and the cardinality of $\Gamma(v)$ is called the {\em valence} of $v$; if $x\in \Gamma(v)$, we also say that $x$ emanates from $v$.
 
If $\Gamma:=(V,D,\beg,\inv)$  is a graph, then we say $\Gamma':=(V',D',\beg',\inv')$ is a {\em subgraph} of $\Gamma$ (and we write $\Gamma' \leq \Gamma$) if $V' \subset V$, $D' \subset D$ and the functions $\beg'$ and $\inv'$ are the respective restrictions of $\beg$ and $\inv$ to $D'$. If additionally $V' = V$ then we say $\Gamma'$ is a {\em spanning subgraph} of $\Gamma$. 

A {\em walk} of {\em length} $n$ is a sequence $(x_1,x_2,\ldots,x_n)$ where $x_i\in \D(\Gamma)$ for all $i\in \{1,\ldots,n\}$, and $\term x_i = \beg x_{i+1}$ for all $i \in \{1,\ldots,n-1\}$. 
If $\beg x_1 = u$ and $\term x_n =v$, then we say $u$ is the {\em initial vertex}, $v$ is the {\em final vertex} and the walk is called a $uv$-walk. A walk of length $0$ is defined as an empty sequence of darts and can be viewed as a $vv$-walk for every vertex $v$.
The {\em inverse} of a walk $W=(x_1,x_2,\ldots,x_n)$ is the walk $W^{-1}=(x_n^{-1},x_{n-1}^{-1},\ldots,x_1^{-1})$. If $W_1$ and $W_2$ are two walks, then we denote by $W_1W_2$ the {\em concatenation} of both sequences. Note that this is well defined only when the final vertex of $W_1$ is equal to the initial vertex of $W_2$. A walk $W$ is {\em reduced} if it contains no two consecutive darts that are inverse to each other. A walk $(x_1,x_2,\ldots,x_n)$ is a {\em path} if $\beg x_i \neq \beg x_j$ for all $0 \leq i,j \leq n$; a {\em closed walk} if $\beg x_1 = \term x_n$ and a {\em cycle} if it is a closed walk and $\beg x_i \neq \beg x_j$ for all  $0 \leq i,j \leq n$. A graph is connected if for any two vertices $u$ and $v$, there exists a $uv$-walk. A connected graph that contains no cycles is a {\em tree}. Note that a tree is always a simple graph.

A {\em morphism} of graphs $\varphi \colon \Gamma \to \Delta$
 is a function $\varphi: \V(\Gamma) \cup \D(\Gamma) \to \V(\Delta) \cup \D(\Delta) $ that maps vertices to vertices, darts to darts and such that 
 $\varphi( \inv _\Gamma x) = \inv_\Delta \varphi(x)$  and
$\varphi(\beg_\Gamma x) = \beg_\Delta \varphi(x)$.
Due to the latter condition, a morphism of graphs without isolated vertices (vertices of degree $0$)
is uniquely determined by its restriction to the set of darts; we will often exploit this fact and
define a morphism on darts only.

A surjective morphism  is called an {\em epimorphism} and a bijective morphism is called an {\em isomorphism}. An isomorphism of a graph onto itself is called an {\em automorphism}.
 
\subsection{Generalised covering graphs}
 
We will now discuss the generalised covering graphs.
We first prove that Definition~\ref{def:gencov} indeed yields a graph.

\begin{lemma}
\label{lem:OK}
Assume the notion from Definition~\ref{def:gencov}. Then
the functions $\beg_{\Gamma}$ and $\inv_{\Gamma}$ are well defined and $\inv_{\Gamma} \inv_{\Gamma}   X= X$ for every dart $X$ of $\Gamma$.
\end{lemma}

\begin{proof}
Suppose that for some dart $x\in \D(\Delta)$ and $g,h\in G$, we have $\omega(x)g = \omega(x)h$. Then $gh^{-1} \in \omega(x) \leq \omega(\beg_{\Delta} x)$, hence $\omega(\beg x)g = \omega(\beg x)h$. This shows that the value $\beg_\Gamma(x,\omega(x)g)$
is independent of the choice of the representative $g$ of the coset $\omega(x)g$ and hence
$\beg_\Gamma$ is a well-defined function on $\D(\Gamma)$.

Similarly $(\zeta(x)g)(\zeta(x)h)^{-1} = (gh^{-1})^{\zeta(x)^{-1}} \in \omega(x)^{\zeta(x)^{-1}} = \omega(x^{-1})$. Therefore $\omega(x^{-1})\zeta(x)g = \omega(x^{-1})\zeta(x)h$, implying that $\inv_\Gamma$ is a well-defined function on $\D(\Gamma)$.

Finally let us show that $\inv_{\Gamma}$ is an involution. Let $X:= (x, \omega(x)g)$ be an
arbitrary dart of $\Gamma$. Then
$$
\Inv_{\Gamma} \Inv_{\Gamma} X = \Inv_{\Gamma} (\Inv_{\Gamma} (x, \omega(x)g)) =
 \Inv_{\Gamma} (x^{-1},  \omega(x^{-1})\zeta(x)  g) =
 $$
 $$
= (x, \omega(x)\zeta(x^{-1}) \zeta(x) g )
= (x, \omega(x)g).
$$
\end{proof}

To illustrate the generalised covering construction we provide two simple examples, both with
the base graph $\Delta$ being the complete graph on two vertices and the group $G$ being
the symmetric group $S_6$.

\begin{example}
\label{ex:S6}
Consider the symmetric group $\mathcal{S}_6$ and set $\sigma = (123)(546)$ and $\rho = (23)(45)$. Let $G = \langle \sigma, \rho \rangle$, $H = \langle \rho \rangle$ and $K = \langle \rho\sigma \rangle$. Let $\Delta$ be a graph consisting of two vertices $u$ and $v$ joined by a single edge between them with $\beg x = u$ and $\beg y = v$. Let $\omega$ be a weight function for $\Delta$ given by $\omega(u)=H$, $\omega(v)=K$, $\omega(x)=\omega(y)={1}$. Let $\zeta$ be the voltage assignment given by $\zeta(x)=\sigma$ and $\zeta(y)=\sigma^2$. Then $(\Delta,G,\omega,\zeta)$ is a voltage graph and its generalised cover is isomorphic to three pairs of parallel edges (see Figure \ref{fig:examples}, right). In contrast, the generalised cover of $(\Delta,G,\omega,\zeta')$ where $\zeta$ assigns trivial voltage to both darts $x$ and $y$ is isomorphic to a cycle of length $6$ (see Figure \ref{fig:examples}, left). 
\begin{figure}[h!]
\label{fig:examples}
\resizebox{\textwidth}{!}{
\begin{tabular}{ccc}
\begin{tikzpicture}
[auto, inner sep=0.5mm, label distance = 1mm]

\node (5) at ( 30:2cm) [vertv] {};
\node (1) at ( 90:2cm) [vertu] {};
\node (4) at ( 150:2cm) [vertv] {};
\node (3) at ( 210:2cm) [vertu] {};
\node (6) at (270:2cm) [vertv] {};
\node (2) at (330:2cm) [vertu] {};

\node (a) at ( 60:1.73cm) [] {}
edge [dentro] node{\tiny $(y,\{1\})$} (5)
edge [dentro] node[swap]{\tiny $(x,\{1\})$} (1);
\node (b) at ( 120:1.73cm) [blank] {}
edge [dentro] node{\tiny $(x,\{1\}\rho)$} (1)
edge [dentro] node[swap]{\tiny $(y,\{1\}\rho)$}(4);
\node (c) at ( 180:1.73cm) [blank] {}
edge [dentro] node[ yshift = -1mm]{\tiny $(y,\{1\}\sigma^2)$} (4)
edge [dentro] node[swap, yshift=1mm]{\tiny $(x,\{1\}\sigma^2)$} (3);
\node (d) at ( 240:1.73cm) [blank] {}
edge [dentro] node{\tiny $(x,\{1\}\rho\sigma^2)$} (3)
edge [dentro] node[swap]{\tiny $(y,\{1\}\rho\sigma^2)$} (6);
\node (e) at (300:1.73cm) [blank] {}
edge [dentro] node{\tiny $(y,\{1\}\sigma)$} (6) 
edge [dentro] node[swap]{\tiny $(x,\{1\}\sigma)$} (2);
\node (f) at (360:1.73cm) [blank] {}
edge [dentro] node[yshift=1mm]{\scriptsize $(x,\{1\}\rho\sigma)$} (2) 
edge [dentro] node[swap, yshift = -1mm]{\tiny $(y,\{1\}\rho\sigma)$} (5);

\node [above, yshift=1mm] at (1) {$(u,H)$};
\node [right, xshift = 1mm] at (5) {$(v,K)$};
\node [right, xshift = 1mm] at (2) {$(u,H\sigma)$};
\node [below, yshift = -1mm] at (6) {$(v,K\sigma)$};
\node [left, xshift = -1mm] at (3) {$(u,H\sigma^2)$};
\node [left, xshift = -1mm] at (4) {$(v,K\sigma^2)$};

\end{tikzpicture}
&&
\begin{tikzpicture}
[auto, inner sep=0.5mm, bend angle = 15]

\node (5) at ( 30:2cm) [vertv] {};
\node (1) at ( 90:2cm) [vertu] {};
\node (4) at ( 150:2cm) [vertv] {};
\node (3) at ( 210:2cm) [vertu] {};
\node (6) at (270:2cm) [vertv] {};
\node (2) at (330:2cm) [vertu] {};

\node (a) at ( 60:1.53cm) [blank] {}
edge [dentro, bend right] node[swap]{\tiny $(y,\{1\}\sigma\rho)$} (5)
edge [dentro, bend left] node{\tiny $(x,\{1\}\rho)$} (1);
\node (b) at ( 60:1.93cm) [blank] {}
edge [dentro, bend right] node[swap]{\tiny $(x,\{1\})$}(1)
edge [dentro, bend left] node{\tiny $(y,\{1\}\sigma)$} (5);
\node (c) at ( 180:1.53cm) [blank] {}
edge [dentro, bend right] node[swap, yshift = -2mm]{\tiny $(y,\{1\})$} (4)
edge [dentro, bend left] node[yshift = 2.5mm]{\tiny $(x,\{1\}\sigma^2)$} (3);;
\node (d) at ( 180:1.93cm) [blank] {}
edge [dentro, bend right] node[swap, yshift = 2mm]{\tiny $(x,\{1\}\rho\sigma^2)$}(3)
edge [dentro, bend left] node[yshift = -2mm] {\tiny $(y,\{1\}\rho\sigma)$} (4);
\node (e) at (300:1.53cm) [blank] {}
edge [dentro, bend right] node[swap]{\tiny $(y,\{1\}\sigma^2)$} (6)
edge [dentro, bend left] node{\tiny $(x,\{1\}\sigma)$} (2);
\node (f) at (300:1.93cm) [blank] {}
edge [dentro, bend right] node[swap]{\tiny $(x,\{1\}\rho\sigma)$}(2)
edge [dentro, bend left] node{\tiny $(y,\{1\}\rho)$}(6);

\node [above, yshift=1mm] at (1) {$(u,H)$};
\node [right, xshift = 1mm] at (5) {$(v,K\sigma)$};
\node [right, xshift = 1mm] at (2) {$(u,H\sigma)$};
\node [below, yshift = -1mm] at (6) {$(v,K\sigma^2)$};
\node [left, xshift = -1mm] at (3) {$(u,H\sigma^2)$};
\node [left, xshift = -1mm] at (4) {$(v,K)$};
\end{tikzpicture}
\end{tabular}}
\bigskip

\begin{tikzpicture}
[auto, inner sep=0.5mm, label distance = 1mm]
\node (5) at ( 0:1cm) [vertv] {};
\node (1) at ( 180:1cm) [vertu] {};

\node (a) at ( 0:0cm) [blank] {}
edge [dentro] node[swap]{$y$} (5)
edge [dentro] node{$x$} (1);
\node [above, yshift = 1mm] at (1) {$u$};
\node [above, yshift = 1mm] at (5) {$v$};
\end{tikzpicture}
\caption{Two covers}
\end{figure}
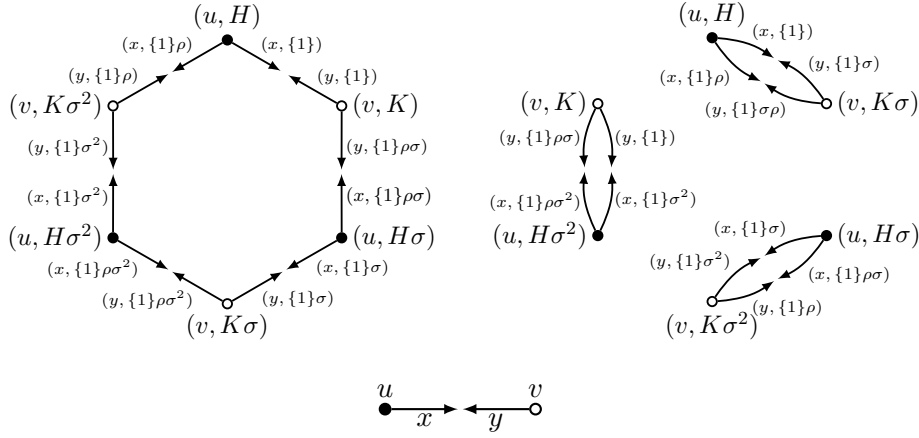
\end{example}

 \begin{remark}
The following formula for the end of a dart in a generalised cover $\Gamma=\GC(\Delta,G,\omega,\zeta)$ will be used often in calculations:
 $$\term_\Gamma (x, \omega(x)g) = 
 \beg_{\Gamma}(x^{-1}, \omega(x^{-1})\zeta(x) g) = (\term_\Delta x, \omega(\term_\Delta x)\zeta(x)g).$$
\end{remark}

\begin{lemma}
\label{lem:inverse}
Let $(\Delta,G,\omega,\zeta)$ be a generalised voltage graph and let $x_0 \in \D(\Delta)$ be a dart such that $x_0 \neq x_0^{-1}$. Then $\GC(\Delta,G,\omega,\zeta) =  \GC(\Delta,G,\omega,\lin{\zeta})$ where
$$
\lin{\zeta}(x) = \left\{
        \begin{array}{ll}        	
            \zeta(x) & \quad \hbox{ if } x \neq x_0^{-1};\\
            \zeta(x_0)^{-1} & \quad \hbox{ if } x = x_0^{-1}.
        \end{array}
    \right.
$$
\end{lemma}

\begin{proof}
Let us begin by showing that $(\Delta,G,\omega,\lin{\zeta})$ is a well defined generalised voltage graph. For this we need to show that $\omega$ and $\lin{\zeta}$ satisfy (\ref{eq:volt2}). Since $(\Delta,G,\omega,\zeta)$ is a generalised voltage graph, we have $\omega(x)=\omega(x^{-1})^{\zeta(x)} = \omega(x^{-1})^{\lin{\zeta}(x)}$ whenever $x \neq x_0$. It remains to show that $\omega(x_0^{-1}) = \omega(x_0)^{\lin{\zeta}(x_0^{-1})}$. From (\ref{eq:volt2}) we have $\omega(x_0) = \omega(x_0^{-1})^{\zeta(x_0)}$ and then $\omega(x_0^{-1}) = \omega(x_0)^{\zeta(x_0)^{-1}}$. Since $\zeta(x_0)^{-1} = \lin{\zeta}(x_0^{-1})$ we see that $\omega(x_0^{-1}) = \omega(x_0)^{\lin{\zeta}(x_0^{-1})}$. Therefore $(\Delta,G,\omega,\lin{\zeta})$ is well defined. 

Now, let $\Gamma:=\GC(\Delta,G,\omega,\zeta)$ and $\lin{\Gamma}:= \GC(\Delta,G,\omega,\lin{\zeta})$. Observe that $\V(\Gamma)=\V(\lin{\Gamma})$, $\D(\Gamma)=\D(\lin{\Gamma})$ and $\beg_{\Gamma}=\beg_{\lin{\Gamma}}$. Moreover, $\inv_{\Gamma}$ agrees with $\inv_{\lin{\Gamma}}$ in all darts except, possibly, in those lying  in $\fib(x_0^{-1})$. Hence, it suffices to show that $\inv_{\Gamma}(x_0^{-1},\omega(x_0^{-1})g) = \inv_{\lin{\Gamma}}(x_0^{-1},\omega(x_0^{-1})g)$ for all $g \in G$. In other words, that $(x_0,\omega(x_0)\zeta(x_0^{-1})g)= (x_0,\omega(x_0)\lin{\zeta}(x_0^{-1})g)$. From (\ref{eq:volt3}) we have $\zeta(x_0^{-1}) \in \omega(x_0)\zeta(x_0)^{-1}$, and then $\omega(x_0)\zeta(x_0^{-1}) = \omega(x_0)\zeta(x_0)^{-1}=\omega(x_0)\lin{\zeta}(x_0^{-1})$. Therefore, $\inv_{\Gamma} = \inv_{\lin{\Gamma}}$ and $\Gamma = \lin{\Gamma}$.
  \end{proof}

\begin{remark}
\label{rem:inverse}
In view of Lemma~\ref{lem:inverse} we see that, without changing the generalised cover $\GC(\Delta,G,\omega,\zeta)$, we can always modify $\zeta$ in such a way that
$\zeta(x)^{-1} =\zeta(x^{-1})$ holds for every dart $x$ not underlying a semiedge.
 \end{remark}

We finish this section by a number of general remarks about the generalised covering graphs.
In what follows we assume that  $(\Delta,G,\omega,\zeta)$ is a generalised voltage graph and that $\Gamma = \GC(\Delta,G,\omega,\zeta)$.


\begin{remark}
\label{rem:9}
If $x \in \V(\Delta) \cup \D(\Delta)$,
 then the set $\{ (x, \omega(x)g) : g \in G\} \subseteq \V(\Gamma) \cup \D(\Gamma)$ is called the {\em fibre above} $x$ and is denoted by
 $\fib(x)$.
Note that the mapping $\varphi\colon \V(\Gamma) \cup \D(\Gamma)\to V(\Delta) \cup \D(\Delta)$ defined by $\varphi(x, \omega(x)g) := x$ for every $(x, \omega(x)g) \in V(\Gamma) \cup \D(\Gamma)$ is a graph epimorphism, which we shall call {\em the generalised covering projection associated with $(\Delta,G,\omega,\zeta)$}.
\end{remark}

\begin{remark}
\label{rem:iota}
For $u \in \V(\Delta)$ let $\iota(u) = |G:\omega(u)|$ and observe that $\iota(u) = |\fib(u)|$. Similarly for $x\in \D(\Delta)$ such that $\beg x =u$, let 
\begin{align}
\label{eq:lambda}
\lambda(x)= |\omega(u):\omega(x)|.
\end{align}
Observe that $\iota(u)\lambda(x) = |G:\omega(x)|=| \fib(x)|$. Moreover, for $\tilde{u} \in \fib(u)$ we have $|\Gamma(\tilde{u}) \cap \fib(x)| = \lambda(x)$ and hence the valence of $\tilde{u}$ in $\Gamma$ is $\sum_{x \in \Delta(u)}\lambda(x)$.
\end{remark}

\begin{remark}
If for every vertex $u\in \V(\Delta)$ there exist a constant $c_u$ such that
$\lambda(x) = c_u$  for every dart $x\in \Delta(u)$, then
the generalised covering projection $\varphi \colon \Gamma \to \Delta$ is
{\em a branched covering} as defined \cite{MedNed1,MedNed2}.
\end{remark}

\subsection{Miscellanea}
 
  If $H$ is a subgroup of $G$, then we let $G/H:=\{ Hg :  g\in G\}$ denote the set of right cosets of $H$ in $G$. If $\varphi$ is a permutation of a set $X$, then we denote the image of $x$ by $\varphi$ as $x^\varphi$ and we define the product of two permutations $\varphi$ and $\psi$ of $X$ by 
 $x^{(\varphi\psi)} :=  (x^\varphi)^\psi$. The symmetric group $\Sym(X)$ is thus defined as the set of all permutations of $X$ equipped with such a product. In particular, all the groups will be acting on the sets from right. For a subgroup $H\le G$ we let $\core_G(H) = \cap_{g\in G} H^g$ denote the {\em core} of $H$ in $G$.

\section{Reconstruction}
\label{sec:reconstruction}

This section is devoted to the proof of Theorem~\ref{the:lift}, stated in Section~\ref{sec:intro}.
We will prove a slightly more detailed version of the theorem.

\begin{theorem}
\label{the:lift2}
Let $\Gamma$ be a graph and let $G \leq \Aut(\Gamma)$.
Let $\mathcal{T}_V$ be a transversal of $\V(\Gamma)/G$ and let $\mathcal{T}_D$ be a transversal of $\D(\Gamma)/G$
 such that $x \in \mathcal{T}_D$ implies $\beg x \in \mathcal{T}_V$ (note that such a pair of transversals always exists).
 Set $\mathcal{T}:= \mathcal{T}_D \cup \mathcal{T}_V$. Then every vertex or dart of $\Gamma/G$ can be written uniquely as 
$x^G$ with $x \in \mathcal{T}$. For $x$ in $\mathcal{T}_D$ let $\iota(x)$ be the unique element of $\mathcal{T}_D$ such that $x^{-1} \in \iota(x)^G$.
Let $\zeta: \D(\Gamma/G) \to G$ be a function such that:
\begin{align}
\label{eq:zetacondition}
\iota(x)^{\zeta(x^G)} = x^{-1} \hbox{ for every } x \in \mathcal{T}_D.
\end{align}

 Define $\omega: \V(\Gamma)/G \cup \D(\Gamma)/G \to S(G)$ by letting 
 $$\omega(x^G):=G_x \>\> \hbox{ for every }\>\> x \in \mathcal{T}.$$
Then the quadruple $(\Gamma/G, G, \omega,\zeta)$ is a generalised voltage graph and
there exists an isomorphism between $\Gamma$  and $\GC(\Gamma/G,G,\omega,\zeta)$
which maps every $G$-orbit on $\V(\Gamma) \cup \D(\Gamma)$ bijectively to a fibre of the corresponding generalised covering graph
$\GC(\Gamma/G,G,\omega,\zeta).$
\end{theorem}
	
\begin{proof}
Since we work with three distinct graphs in this proof, namely $\Gamma$, $\Gamma/G$ and $\Theta:=\GC(\Gamma/G,G,\omega,\zeta)$,
we will be very careful not to confuse the corresponding $\V$, $\D$, $\beg$ and $\inv$ operators. In particular, we let 
$(V',D',\beg',\Inv'):=\Gamma/G$ and reserve the shorthand notation $x^{-1}$ for $\inv_\Gamma x$ only when $x\in \D(\Gamma)$.
Further, for $x\in \V(\Gamma) \cup \D(\Gamma)$, we let $\overline{x}:=x^G$ denote the $G$-orbit of $x$; note that then $\lin{x} \in V'\cup D'$.

  We will now show that the quadruple $(\V(\Gamma)/G, \D(\Gamma)/G, \omega,\zeta)$ satisfies conditions (\ref{eq:volt1})--(\ref{eq:volt3}).
For the remainder of this proof, consider a dart $x \in \mathcal{T}_D$. 
By definition of the quotient graph it follows that the darts $\lin{x}$ and $\lin{\iota(x)}$ are mutually inverse; that is
\begin{align}
\label{eq:invxy}
\invP \lin{x}=  \lin{\iota(x)}\>\> \hbox{ and  }\>\> \invP \lin{\iota(x)} = \lin{x}. 
\end{align}
Moreover, by definition of $\zeta$ it also follows that 
$\iota(x) = (x^{-1})^{\zeta(\lin{x})^{-1}}$ and since $\zeta$ is an automorphism, we have $\iota(x)^{-1} = x^{\zeta(\inv' \lin{x})}$ and so
\begin{align}
\label{eq:zeta1}
x^{\zeta(\lin{\iota(x)})} = \iota(x)^{-1}.
\end{align}
Similarly, by (\ref{eq:zetacondition}) we have
\begin{align}
\label{eq:zeta2}
(\iota(x)^{-1})^{\zeta(\lin{x})} = x.
\end{align}
Using formulae (\ref{eq:invxy}), (\ref{eq:zeta1}) and (\ref{eq:zeta2}), we may conclude that
 $$x^{\zeta(\inv' \lin{x})\zeta(\lin{x})}=x^{\zeta(\lin{\iota(x)})\zeta(\lin{x})}=(\iota(x)^{-1})^{\zeta(\lin{x})}=x.$$ 
In particular, $\zeta(\inv' \lin{x})\zeta(\lin{x}) \in G_x$, and since $\omega(\lin{x})=G_x$, we see that the condition (\ref{eq:volt3}) is satisfied.

Let us now show that the condition (\ref{eq:volt2}) is satisfied. Observe first that the stabilisers $G_{\iota(x)}$ and $G_{\iota(x)^{-1}}$ of mutually inverse darts $y$ and $y^{-1}$ are equal and that for any $z\in \D(\Gamma)$ and any $g\in \Aut(\Gamma)$, the conjugate $(G_z)^g$ of the stabiliser $G_z$  equals the stabiliser $G_{z^g}$. We thus see that
$$\omega(\invP \lin{x})^{\zeta(\lin{x})} = \omega(\lin{\iota(x)})^{\zeta(\lin{x})} = 
(G_{\iota(x)})^{\zeta(\lin{x})} =  (G_{\iota(x)^{-1}})^{\zeta(\lin{x})} = $$
$$G_{ (\iota(x)^{-1})^{\zeta(\lin{x})} } =
G_x = \omega(\lin{x}),$$
showing that the condition (\ref{eq:volt2}) is satisfied.
 Finally, since $G_x \leq G_{\beg x}$ for any dart $x\in \D(\Gamma)$, we see that (\ref{eq:volt1}) holds. Therefore $(\Gamma/G, G, \omega, \zeta)$ is a generalised voltage graph.

We will now show that the corresponding generalised cover $\Theta$ 
is isomorphic to the graph $\Gamma$.
By definition of the generalised cover and the function $\omega$, we see that $\V(\Theta) = \{ (\lin{v}, G_vg) : v \in {\mathcal{T}}_V, g \in G\}$ and
$\D(\Theta) = \{ (\lin{x}, G_xg) : x \in {\mathcal{T}}_D, g \in G\}$.
Let $\varphi: \V(\Theta) \cup \D(\Theta) \to \V(\Gamma) \cup \D(\Gamma)$ be given by 
$$\varphi(\lin{x},G_xg)=x^g \> \hbox{ for } \> x \in \mathcal{T}.$$ 

To see that $\varphi$ is well defined, suppose that for some $x \in  \mathcal{T}$, we have $G_xg = G_xg'$. It follows that $g'g^{-1} \in G_x$ and so $x^g=x^{g'}$. Then $\varphi(\lin{x},G_xg) = x^g = x^{g'} = \varphi(\lin{x},G_xg')$. Hence $\varphi$ is well defined. Since the converse of every implication in the preceding lines holds, this also shows that $\varphi$ is injective. Since $\varphi$ is clearly surjective, it remains to be shown that $\beg_\Gamma \varphi(X) = \varphi(\beg_\Theta X)$ and $\Inv_\Gamma \varphi(X) = \varphi(\Inv_\Theta X)$ for every dart  $X\in \D(\Theta)$.
Let $X\in \D(\Theta)$. Then $X=(\lin{x},G_xg)$ for some $x\in \mathcal{T}_D$. Then
$$\beg_\Gamma \varphi(X) =\beg_\Gamma \varphi(\lin{x},G_xg)=\beg_\Gamma (x^g)= (\beg_\Gamma x)^g.$$
 On the other hand
 $$
 \varphi(\beg_\Theta (\lin{x},G_xg))
 = \varphi(\begP \lin{x}, \omega(\begP \lin{x})g) 
= \varphi (\lin{(\beg x)}, G_{\beg x}g) 
= (\beg x)^g.
 $$

Hence $\beg_\Gamma \varphi(X) = \varphi(\beg_\Theta X)$, as claimed. 
Further, we see that 
$$\invG \varphi(X)= \invG \varphi(\lin{x},G_xg) = \invG(x^g)=(\invG x)^g.$$
Recall now the definition of the dart $\iota(x)$ and the fact that $\invP \lin{x} = \lin{\iota(x)}$. Then
 $$ \varphi(\invI(\lin{x},G_xg))= \varphi(\invP \lin{x}, \omega(\invP \lin{x})\zeta(\lin{x})g)
 = \varphi(\lin{\iota(x)},G_{\iota(x)}\zeta(\lin{x})g) = $$ 
 $$\iota(x)^{\zeta(\lin{x})g} = (\invG x)^g.$$
 Note the the last equality follows from (\ref{eq:zeta2}) while the second to last from the fact that $\iota(x) \in \mathcal{T}_D$.
 We have thus shown that $\invG \varphi(X) = (\invG x)^g = \varphi (\inv_\Theta X)$. 
We conclude that $\varphi$ is an isomorphism and thus $\Gamma \cong \Theta$, as claimed.
Moreover, $\varphi$ clearly maps fibres in $\Theta$ to $G$-orbits on $\Gamma$.
This finishes the proof of Theorem~\ref{the:lift}.
\end{proof}

\begin{remark}
\label{rem:trivial}
Assume the notation introduced in the statement and the proof of the above theorem.
Consider a dart $X\in \Gamma/G$ and let $x$ be the unique element  in $\mathcal{T}_D$ such that $X=x^G$.
Then $\zeta(X)$ is defined to be an element $g$ of $G$
which maps $x^{-1}$ to the unique element $\iota(x)\in\mathcal{T}$ such that $(x^{-1})^G = \iota(x)^G$
(note that there is some freedom in the choice of such an element $g$).
In particular, if $x^{-1} \in \mathcal{T}_D$, then one can assume that $\zeta(X)$ is trivial.
\end{remark}

For the reasons that will become apparent in Remark~\ref{rem:faithful} below and in
Proposition~\ref{faithful}, we introduce the following property of generalised voltage graphs:

\begin{definition}
\label{def:faithful}
A generalised voltage graph $(\Delta,G,\omega,\zeta)$ is said to be 
{\em faithful} provided that
\begin{equation}
  \core_G\big(\bigcap_{x\in \D(\Delta)} \omega(x)\big) = 1.
\end{equation}
\end{definition}

\begin{remark}
\label{rem:faithful}
The generalised voltage graph $(\Gamma/G,G,\omega,\zeta)$ mentioned in
the statement of Theorem~\ref{the:lift} is in fact faithful. Indeed: Using the notation introduced
 in the proof above, observe that $\{x^g : x \in {\mathcal T}_D, g \in G\} = D$.
 Therefore
$$
\core_G \bigl(  \bigcap_{X\in \D(\Gamma/G)} \omega(X) \bigr) = 
\core_G \bigl(  \bigcap_{x\in {\mathcal T}_D} G_x \bigr) =
\bigcap_{x\in {\mathcal T}_D, g\in G} (G_x)^g  =
\bigcap_{x\in D(\Gamma)} G_x,
$$
which is trivial since $G$ acts faithfully on $\D(\Gamma)$. 
\end{remark}

\section{Automorphisms of generalised covers}
\label{sec:auto}

This section is devoted to the study of the automorphisms of
the generalised covering graph $\Gamma$ arising from a given
generalised voltage graph $(\Delta,G,\omega,\zeta)$.

The first lemma and the proposition that follows it show that the group $G$ acts in a natural way on $\Gamma$, as one would of course expect.

\begin{lemma}
\label{lem:auto}
Let $(\Delta,G,\omega,\zeta)$ be a generalised voltage graph, let $h\in G$ and
let $\Gamma = \GC(\Delta,G,\omega,\zeta)$. Then the permutation $\bh$ of 
$\V(\Gamma)\cup \D(\Gamma)$ defined by $(x,\omega(x)g)^ \bh:=(x,\omega(x)gh)$ for every $x\in \V(\Delta)\cup \D(\Delta)$ and $g\in G$ is an automorphism of $\Gamma$.
\end{lemma}

\begin{proof}
We leave it to the reader to check that $\bh$ is indeed a permutation of $\V(\Gamma)\cup \D(\Gamma)$. To see that $\bh$ is an automorphism of $\Gamma$, observe that:
\begin{eqnarray*}
(\beg (x,\omega(x)g))^\bh &=& (\beg x, \omega(\beg x)g)^\bh\\
&=&(\beg x, \omega(\beg x)gh) \\
&=& \beg(x,\omega(x)gh) \\
&=& \beg((x,\omega(x)g)^\bh), 
\end{eqnarray*}
and
\begin{eqnarray*}
((x,\omega(x)g)^{-1})^\bh &=& (x^{-1},\omega(x^{-1})\zeta(x)g)^\bh \\
&=& (x^{-1},\omega(x)^{-1}\zeta(x)gh) \\
&=& (x,\omega(x)gh)^{-1}\\
&=& ((x,\omega(x)g)^\bh)^{-1}.
\end{eqnarray*}
Therefore, $\bh$ is a graph morphism from $\Gamma$ to $\Gamma$ as it commutes with $\beg$ and $\inv$, and thus an automorphism of $\Gamma$
\end{proof}

Recall that we call the generalised voltage graph $(\Delta,G,\omega,\zeta)$ is faithful
provided that the core of the group $\cap_{x\in \D(\Delta)} \omega(x)$ in $G$ is trivial
(see Definition~\ref{def:faithful}).
We now prove a result that explains the choice of the word ``faithful''.

\begin{proposition}
\label{faithful}
Let $(\Delta,G,\omega,\zeta)$ be a generalised voltage graph.
Then the mapping $\Phi:G \to \Aut(\Gamma)$, $\Phi(h)=\bh$, is
a homomorphism of groups, which is injective if and only if
the generalised voltage graph is faithful.
\end{proposition}

\begin{proof}
Let $K$ be the kernel of $\Phi$.
Observe that for $h \in G$, $(x,\omega(x)g)^\bh=(x,\omega(x)g)$ if and only if $h^{g^{-1}} \in \omega(x)$ or, equivalently, $h \in \omega(x)^g$. Hence, $h \in K$ if and only if 
\begin{equation*}
h \in \bigcap\limits_{x \in \D(\Gamma)\cup\V(\Gamma)} \bigcap\limits_{g \in G} \omega(x)^g. 
\end{equation*}

However, since $\omega(x) \leq \omega(\beg x)$ for all $x \in \D(\Delta)$, we see that $\omega(x)^g \leq \omega(\beg x)^g$ and thus
\begin{equation}
\label{eq:core}
\bigcap\limits_{x \in \D(\Gamma)\cup\V(\Gamma)} \bigcap\limits_{g \in G} \omega(x)^g = \bigcap\limits_{x \in \D(\Gamma)} \bigcap\limits_{g \in G} \omega(x)^g = 
\bigcap\limits_{g \in G} \bigl( \bigcap\limits_{x \in \D(\Gamma)}  \omega(x)\bigr)^g. 
\end{equation}
Therefore, $K$ is trivial if and only if the right side of (\ref{eq:core}) is trivial, that is, if and only if the generalised voltage graph is faithful.
\end{proof}

The next lemma and the corollaries that follow it establish another natural source of
automorphisms of the generalised covering graph.

\begin{lemma}
\label{lem:isovolt}
Let $(\Delta,G,\omega,\zeta)$ and $(\Delta',G',\omega',\zeta')$ be two generalised voltage graphs, 
and let $\Gamma = \GC(\Delta,G,\omega,\zeta)$ and $\Gamma' = \GC(\Delta',G',\omega',\zeta')$.
Suppose that there exist a graph isomorphisms $\varphi:\Delta \to \Delta'$ and a group isomorphism $f:G \to G'$ such that $f(\omega(x))=\omega'(\varphi(x))$ for every $x\in \V(\Delta) \cup \D(\Delta)$ and $f(\zeta(x))=\zeta'(\varphi(x))$ for every $x\in \V(\Delta)$. 
Let  $\Phi: \V(\Gamma) \cup \D(\Gamma) \to \V(\Gamma') \cup \D(\Gamma')$ be given by 
$$
\Phi \colon (x,\omega(x)g) \mapsto (\varphi(x),\omega'(\varphi(x))f(g)) \hbox{ for every } x \in \V(\Delta) \cup \D(\Delta).
$$
Then $\Phi$ is an isomorphism of graphs between $\Gamma$ and $\Gamma'$.
\end{lemma}

\begin{proof}
We first show that the image of $\Phi$ does not depend on the choice of representatives of a coset $\omega(x)g$ and that $\Phi$ is injective. Suppose that $\omega'(\varphi(x))f(g)) = \omega'(\varphi(x))f(h))$ for some $h,g\in G$. Then $f(\omega(x))f(g) = f(\omega(x))f(h)$ and hence, since $f$ is a group isomorphism,
 $\omega(x)g=\omega(x)h$.  This shows that $\Phi$ is a well-defined function on $\V(\Gamma) \cup \D(\Gamma)$ and that it is injective. Surjectivity of $\Phi$ follows directly from the fact that both $\varphi$ and $f$ are surjective. It remains to see that $\Phi$ intertwines the functions $\beg$ and $\beg'$ as well as $\inv$ and $\inv'$. Note that
$$\Phi(\beg(x,\omega(x)g)) = \Phi(\beg x,\omega(\beg x)g) = (\varphi(\beg x),f(\omega(\beg x))f(g)).$$
But $\varphi$ is a graph isomorphism, therefore 
$$
\varphi(\beg x) = \beg'(\varphi(x))\> \hbox{ and }\> f(\omega(\beg x)) = \omega'(\varphi(\beg x))=\omega'(\beg' (\varphi(x))).
$$
 Hence
$$
 \Phi(\beg(x,\omega(x)g)) = (\beg'(\varphi(x)),\omega(\beg'(\varphi(x))f(g)).
$$
On the other hand,
 $$
 \beg'(\Phi(x,\omega(x)g))=\beg'(\varphi(x),f(\omega(x))f(g))=(\beg'(\varphi(x)), \omega'(\beg'( \varphi(x)))f(g)).
 $$
 Therefore, $\Phi(\beg(x,\omega(x)g)) = \beg'\Phi(x,\omega(x)g)$, as required.
Similarly, 
$$
\Phi((x,\omega(x)g)^{-1}) = \Phi(x^{-1},\omega(x^{-1})\zeta(x)g)=(\varphi(x^{-1}),f(\omega(x^{-1}))f(\zeta(x)g)).
$$
But $f(\omega(x^{-1}))f(\zeta(x)g))=\omega'(\varphi(x^{-1}))f(\zeta(x))f(g))=\omega'(\varphi(x)^{-1})\zeta'(\varphi(x))f(g)$. Thus, 
$$
\Phi((x,\omega(x)g)^{-1}) = (\varphi(x^{-1}),\omega'(\varphi(x)^{-1})\zeta'(\varphi(x))f(g)).
$$
On the other hand, 
$$
 (\Phi(x,\omega(x)g))^{-1} = (\varphi(x),\omega'(\varphi(x))f(g))^{-1} = (\varphi(x^{-1}), \omega'(\varphi(x)^{-1})\zeta'(\varphi(x))f(g)).
$$
This shows that $\Phi(x,\omega(x))^{-1} = (\Phi(x,\omega(x)g)^{-1})$ and concludes the proof that $\Phi$ is an isomorphism of graph $\Gamma$ and $\Gamma'$.
\end{proof}

This result has two immediate but useful consequences.

\begin{corollary}
\label{cor:aut1}
Suppose $(\Delta,G,\omega,\zeta)$ is a generalised voltage graph and $f$ an automorphism of the group $\Aut(G)$. Then
$$
\GC(\Delta,G,\omega,\zeta)\cong\GC(\Delta,G,f\circ \omega,f\circ \zeta).
$$ 
\end{corollary}

\begin{proof}
Set $\varphi:= \id$, $\omega':=f \circ \omega$, $\zeta':=f \circ \zeta$ and apply Lemma~\ref{lem:isovolt}.
\end{proof}

\begin{corollary}
\label{cor:aut2}
Let $(\Delta,G,\omega,\zeta)$ be a generalised voltage graph,
let $\Gamma=\GC(\Delta,G,\omega,\zeta)$,
and let $\varphi \in \Aut(\Delta)$ and $f\in \Aut(G)$ be such that
$\omega(x^\varphi) = f(\omega(x))$ for every $x\in \V(\Delta) \cup \D(\Delta)$ and
$\zeta(x^\varphi) = f(\zeta(x))$ for every $x\in \V(\Delta)$. 
Then the permutation $\Phi$ given by
$$ 
 (x,\omega(x)g)^\Phi =  (x^\varphi,\omega(x^\varphi)f(g)) \hbox{ for every } x \in \V(\Delta) \cup \D(\Delta)
$$
is an automorphism of $\Gamma$.
\end{corollary}

\begin{remark}
Both the automorphisms arising from Lemma~\ref{lem:auto} as well as those from
Corollary~\ref{cor:aut2} preserve the partition of the vertices and darts of the generalised
covering graph into fibres. What is more, the automorphisms arising from Lemma~\ref{lem:auto}
preserve each fibre set-wise. In the case of (non-generalised) covering graphs, there
exists a nice characterisation of such automorphisms of the covering graph; see \cite{lift}, for example.
In the case of generalised covering graphs, such a nice characterisation seems to be out of the reach.
\end{remark}

\section{Normalised voltages}
\label{sec:normal}

It this section we first show that the generalised voltage assignment $\zeta$ can always be assumed to
be trivial on the darts of any fixed spanning tree of the base graph $\Delta$. This is the analogue of
a well-known result in the theory of usual (non-generalised) voltage graphs (see, \cite[p.~91]{grosstuc}, for example).

\begin{lemma}
\label{lem:trivialdart}
Let $(\Delta,G,\omega,\zeta)$ be a generalised voltage graph and let $x \in \D(\Delta)$ such that $\beg x \not =  \term x $.
 Then there exists a voltage assignment $\zeta_x$ assigning the trivial element $1_G$ to $x$ and a weight function $\omega_x$ such that $\GC(\Delta,G,\omega,\zeta) \cong \GC(\Delta,G,\omega_x,\zeta_x)$.
\end{lemma}


\begin{proof}
Let $v := \term x$ and note that by assumption, $\beg x \not = v$. Let  $\beg:=\beg_\Delta$.
Now define $\zeta_x: \D(\Delta) \to G$ and $\omega_x: \D(\Delta) \cup \V(\Delta) \to S(G)$ as follows: 
$$
\zeta_x(y) = \left\{
        \begin{array}{ll}        	
            \zeta(y)\zeta(x) & \quad \hbox{ if } \beg y = v;\\
            \zeta(x)^{-1}\zeta(y) & \quad \hbox{ if } \term y = v;\\
            \zeta(y) & \quad \text{otherwise}.
        \end{array}
    \right.
$$
$$
\omega_x(z) = \left\{
        \begin{array}{ll}        	
            \omega(z)^{\zeta(x)} & \quad \hbox{ if } v \in \{z, \beg z\};\\
            \omega(z) & \quad \text{otherwise}.
        \end{array}
    \right.
$$
Let us show that $(\Delta,G,\omega_x,\zeta_x)$ is a generalised voltage graph. Clearly $\omega_x(y) \leq \omega_x(\beg y)$ for all $y \in \D(\Delta)$, so condition (\ref{eq:volt1}) holds. 

Let us show that (\ref{eq:volt2}) holds, that is, that $\omega_x(y) = \omega_x(y^{-1})^{\zeta_x(y)}$ for every $y\in \D(\Delta)$.
If $\beg y =v$, then $\omega_x(y^{-1})^{\zeta_x(y)} = \omega(y^{-1})^{\zeta(y)\zeta(x)} = \omega(y)^{\zeta(x)} = \omega_x(y)$.
Suppose now that $\term y= v$.  Since $\beg y^{-1} =v$, we see that $\omega_x(y^{-1})=\omega(y^{-1})^{\zeta(x)}$.
Therefore $\omega_x(y^{-1})^{\zeta_x(y)} = \omega_x(y^{-1})^{\zeta(x)^{-1}\zeta(y)} = \omega(y^{-1})^{\zeta(y)}$.
Finally, if neither $\beg y$ nor $\term y$ equals $v$, then $\omega_x(y^{-1})^{\zeta_x(y)} = \omega(y^{-1})^{\zeta(y)} = \omega(y) = \omega_x(y)$.
We thus see that (\ref{eq:volt2}) holds for every $y\in \D(\Delta)$.

Now, let us show that the condition (\ref{eq:volt3}) is satisfied, that is, that for every $y\in \D(\Delta)$, we have
$\zeta_x(y^{-1})\zeta_x(y) \in \omega_x(y)$.
 If $\beg y = v$, then $\zeta_x(y^{-1})\zeta_x(y) = \zeta(x)^{-1}\zeta(y^{-1})\zeta(y)\zeta(x)$. But $\zeta(y^{-1})\zeta(y) \in \omega(y)$, 
 therefore $\zeta_x(y^{-1})\zeta_x(y) \in \omega(y)^{\zeta(x)}$, and since $\omega(y)^{\zeta(x)} = \omega_x(y)$, 
 we see that $\zeta_x(y^{-1})\zeta_x(y) \in \omega_x(y)$. 
 If $\term y = v$, then $\zeta_x(y^{-1})\zeta_x(y) = \zeta(y^{-1})\zeta(x)\zeta(x)^{-1}\zeta(y) = \zeta(y^{-1})\zeta(y)$. Since $\zeta(y^{-1})\zeta(y) \in \omega(y)$ and $\omega(y) = \omega_x(y)$, we see that $\zeta(y^{-1})\zeta(y) \in \omega_x(y)$. 
Finally, if neither $\beg y$ nor $\term y$ equals $v$, then $\zeta_x(y^{-1})\zeta_x(y) = \zeta(y^{-1})\zeta(y) \in \omega(y) = \omega_x(y)$.
 This concludes the proof of  (\ref{eq:volt3}) and shows that $(\Delta,G,\omega_x,\zeta_x)$ is a generalised voltage graph.     

Now let $\Gamma := \GC(\Delta,G,\omega,\zeta)$ and $\Gamma' := \GC(\Delta,G,\omega_x,\zeta_x)$.
 Define $\varphi: \V(\Gamma) \cup \D(\Gamma) \to \V(\Gamma') \cup \D(\Gamma')$ by
$$
\varphi((z,\omega(z)g)) = \left\{
        \begin{array}{ll}
        	(z,\omega_x(z)\zeta(x)^{-1}g) &  \> \hbox{ for } \> z  \in \V(\Gamma) \cup \D(\Gamma), \> v\in \{z, \beg z\}; \\
            (z,\omega_x(z)g) & \>  \hbox{ for } \> z \in \V(\Gamma) \cup \D(\Gamma) ,\> v\not \in \{z, \beg z\}.
        \end{array}
    \right.
$$

We first show that $\varphi$ is well defined and injective. 
Let $a, b \in G$ and let $z \in \V(\Delta) \cup \D(\Delta)$.
We need to show that $\varphi(z,\omega(z)a) = \varphi(z,\omega(z)b)$ if only if $(z,\omega(z)a) = (z,\omega(z)b)$.
If $v\not \in \{z, \beg z\}$, then $\varphi(z,\omega(z)a) = (z,\omega_x(z)a) =(z,\omega(z)a) $ and
$\varphi(z,\omega(x)b) = (z,\omega_x(z)b)=(z,\omega(z)b)$ and the claim follows.
Assume now that $v\in \{z, \beg z\}$. Note that by the definition of $\omega_x$, 
it then follows that $\omega_x(z)=\omega(z)^{\zeta(x)}$.
Hence $ab^{-1} \in \omega(z)$ holds if and only if $ab^{-1} \in \omega_x(z)^{\zeta(x)^{-1}}$ holds, and therefore
 $\omega(z)a = \omega(z)b$ holds  if and only if $\omega_x(z)^{\zeta(x)^{-1}}a = \omega_x(z)^{\zeta(x)^{-1}}b$ holds
 if and only if  $\omega_x(z){\zeta(x)^{-1}}a = \omega_x(z){\zeta(x)^{-1}}b$ holds,
  and thus if and only if $\varphi(z,\omega(x)a) = \varphi(z,\omega(x)b)$ holds.
We have thus shown that $\varphi$ is well defined and injective.

To prove surjectivity of $\varphi$, take an arbitrary $(z,\omega_x(z)g) \in \V(\Gamma')\cup \D(\Gamma')$. If
$v\not \in \{z, \beg z\}$, then $(z,\omega_x(z)g) = \varphi(z,\omega(z)g)$ and thus $(z,\omega_x(z)g)$ is in the image of $\varphi$.
Similarly, if $v\in \{z, \beg z\}$, then $(z,\omega_x(z)g) = \varphi(z,\omega(z)\zeta(x)^{-1}g)$, proving that $\varphi$ is surjective.

It remains to show that $\varphi$ intertwines the functions $\beg_\Gamma$ and $\beg_{\Gamma'}$ as well as the functions
$\inv_\Gamma$ and $\inv_{\Gamma'}$.
Let us first show that $\varphi$ intertwines $\beg_\Gamma$ and $\beg_{\Gamma'}$.
If  $\beg y = v$, then:
\begin{eqnarray*}
\varphi(\beg_\Gamma (y,\omega(y)g))&=& \varphi((\beg y,\omega(\beg y)g))\\
&=& (\beg y,\omega_x(\beg  y)\zeta(x)^{-1}g)\\
&=& \beg_{\Gamma'} (y,\omega_x(y)\zeta(x)^{-1}g) \\
&=& \beg_{\Gamma'}(\varphi(y,\omega(y)g)).
\end{eqnarray*}

On the other hand, if $\beg y \not = v$, then:
\begin{eqnarray*}
\varphi(\beg_\Gamma(y,\omega(y)g)) &=& \varphi((\beg y,\omega(\beg y)g)) \\
&=& (\beg y,\omega_x(\beg y)g)\\
&=& \beg_{\Gamma'}(y,\omega_x(y)g)\\
&=& \beg_{\Gamma'}(\varphi(y,\omega(y)g)).
\end{eqnarray*}

Let us now show that $\varphi$ intertwines 

 $\inv_\Gamma$ and $\inv_{\Gamma'}$. 
If  $\term y =v$, then:
\begin{eqnarray*}
\varphi((y,\omega(y)g)^{-1}) &=& \varphi((y^{-1},\omega(y^{-1})\zeta(y)g)) \\
&=& (y^{-1},\omega_x(y^{-1})\zeta(x)^{-1}\zeta(y)g)\\
&=& (y^{-1}, \omega_x(y^{-1})\zeta_x(y)g)\\
&=& (y, \omega_x(y)g)^{-1}\\
&=& \varphi(y,\omega(y)g)^{-1}.
\end{eqnarray*}

Similarly, if  $\beg y  =v$, then:
\begin{eqnarray*}
\varphi((y,\omega(y)g)^{-1}) &=& \varphi((y^{-1},\omega(y^{-1})\zeta(y)g)) \\
&=& (y^{-1},\omega_x(y^{-1})\zeta(y)g)\\
&=& (y^{-1}, \omega_x(y^{-1})\zeta_x(y)\zeta(x)^{-1}g)\\
&=& (y, \omega_x(y)\zeta(x)^{-1}g)^{-1}\\
&=& \varphi(y,\omega(y)g)^{-1}.
\end{eqnarray*}

Finally, if  $v\not\in \{\beg y, \term y\}$, then:
\begin{eqnarray*}
\varphi((y,\omega(y)g)^{-1}) &=& \varphi((y^{-1},\omega(y^{-1})\zeta(y)g)) \\
&=& (y^{-1},\omega_x(y^{-1})\zeta(y)g)\\
&=& (y^{-1}, \omega_x(y^{-1})\zeta_x(y)g)\\
&=& (y, \omega_x(y)g)^{-1}\\
&=& \varphi(y,\omega(y)g)^{-1}.
\end{eqnarray*}
 
Therefore $\varphi$ is an isomorphism and $\Gamma \cong \Gamma'$, as claimed.
\end{proof}

\begin{remark}
\label{rem:volt1}
The voltage assignments $\zeta$ and $\zeta_x$ of Lemma \ref{lem:trivialdart} coincide in all darts of $\Delta$, except in those whose initial vertex is $v$.
\end{remark}

\begin{theorem}
\label{lem:normalized}
Let $(\Delta,G,\omega,\zeta)$ be a voltage graph and let $T$ be a spanning tree of $\Delta$. Then there exists a voltage assignment $\zeta'$ and a weight function $\omega'$ such that $\zeta'(x)=1_G$ for all $x \in \D(T)$ and $\GC(\Delta,G,\omega,\zeta) \cong \GC(\Delta,G,\omega',\zeta')$.
\end{theorem}

\begin{proof}
Let $(\Delta,G,\omega,\zeta)$ be a voltage graph. We will show that for any subgraph $T$ of $(\Delta,G,\omega,\zeta)$ that is isomorphic to a tree, there is a voltage assignment $\zeta^{T}$ and a weight function $\omega^{T}$ such that $\zeta^{T}$ is trivial on all darts of $T$ and $\GC(\Delta,G,\omega,\zeta) \cong \GC(\Delta,G,\omega^T,\zeta^T)$. We proceed by induction on the size (the number of edges) of $T$. Note that if $T$ has only one edge, then the result follows at once from Lemma \ref{lem:trivialdart} and Remark~\ref{rem:inverse}. Now, suppose the result holds for any tree of size $k$, $k<|\V(\Delta)|-1$. Let $T$ be a spanning tree of $\Delta$ and notice that $T$ has $|\V(\Delta)|-1$ edges. Let $v$ be a leaf (a vertex of valency $1$) of $T$ and let $u$ be the unique vertex of $T$ adjacent to $v$. Then there is a unique dart $x$ in $T$ such that $\beg x = u$ and $\term x = v$. Let $T'$ be the tree obtained by deleting the vertex $v$ from $T$. Then, by induction hypothesis there is a weight function $\omega^{T'}$ and a voltage assignment $\zeta^{T'}$ such that $\zeta^{T'}$ is trivial on all dart in $T'$ and $\GC(\Delta,G,\omega,\zeta) \cong \GC(\Delta,G,\omega^{T'},\zeta^{T'})$. By Lemma \ref{lem:trivialdart}, there exists a weight function $\omega^T$ and a voltage assignment $\zeta^T$ such that $\zeta^T(x) =  1_G$ and $\GC(\Delta,G,\omega^{T'},\zeta^{T'}) \cong \GC(\Delta,G,\omega^{T},\zeta^{T})$. Moreover, by Remark \ref{rem:volt1}, $\zeta^T(y) = \zeta^{T'}(y)$ for all $y \in \D(T)\setminus\{x,x^{-1}\}$. Therefore, $\zeta^T$ is trivial on all darts of $T$ and $\GC(\Delta,G,\omega,\zeta) \cong \GC(\Delta,G,\omega^{T'},\zeta^{T'}) \cong \GC(\Delta,G,\omega^{T},\zeta^{T})$.
\end{proof}

\begin{remark}
\label{rem:Tnormal}
Let $T$ be a spanning tree of a graph $\Delta$. A generalised voltage graph $(\Delta,G,\omega,\zeta)$ where $\zeta(x) = 1$ for every dart of $T$ is said to be {\em $T$-normalised}. In light of Theorem \ref{lem:normalized} we can always assume a voltage graph to be $T$-normalised for a prescribed spanning tree $T$. This will often prove to be useful as it makes calculations simpler.
\end{remark}

\begin{remark}
\label{rem:fib}
Observe that the isomorphism between $\GC(\Delta,G,\omega,\zeta)$ and $\GC(\Delta,G,\omega_x,\zeta_x)$ provided in the proof of Lemma~\ref{lem:trivialdart} mapped fibres of the first generalised cover to the fibres of the second generalised cover. Consequently, also the isomorphism
between the graphs $\GC(\Delta,G,\omega,\zeta)$ and $\GC(\Delta,G,\omega',\zeta')$ in
Theorem~\ref{lem:normalized} can be chosen to map fibres to fibres.
\end{remark}

\section{Connectivity}
\label{sec:con}

The aim of this section is a characterisation of connectivity of the generalised covering graph $\GC(\Delta,G,\omega,\zeta)$ in terms of the generalised voltage graph $(\Delta,G,\omega,\zeta)$. The main result of the section is Theorem~\ref{lem:connect}.

\begin{lemma}
\label{lem:neighbours}
Let $\Gamma=\GC(\Delta,G,\omega,\zeta)$ and let $(u,\omega(u)g)$ be a vertex of $\Gamma$. Then the set of neighbours of $(u,\omega(u)g)$ is the set $\{(\term_\Delta x,\omega(\term_\Delta x)zg) \mid z \in \zeta(x)\omega(u), \beg_\Delta x =u \}$.
\end{lemma}

\begin{proof}
 Notice that the neighbourhood of $(u,\omega(u)g)$ is precisely $\{ \term_\Gamma (x,\omega(x)h) \mid \beg (x,\omega(x)h) = (u,\omega(u)g)\}$. Now, $\beg_\Gamma (x,\omega(x)h) = (u,\omega(u)g)$ if and only $\beg_\Delta x =u$ and $h \in \omega(u)g$. Since $\term_\Gamma (x,\omega(x)h) = (\term_\Delta x, \omega(\term_\Delta x)\zeta(x)h)$, we see that the neighbourhood of $(u,\omega(u)g)$ consists of all vertices of the form $(\term_\Delta x,\omega(\term_\Delta x)zg)$ where $z \in \zeta(x)\omega(u)$ and $\beg_\Delta x = u$.
\end{proof}

Let $\Gamma=\GC(\Delta,G,\omega,\zeta)$ 
 and let $W=(x_0,x_1,...,x_{n-1})$ be a walk of length $n$ in $\Delta$. Let $v_i = \beg_\Delta(x_i)$,  for $i \in \{0,...,n-1\}$, and let $v_n = \term(x_{n-1})$. We define the {\em voltage of $W$}, denoted $\zeta(W)$, as
 \begin{equation}
 \label{eq:zetaW}
\zeta(W) = \prod_{i=0}^{n-1} \zeta(x_{n-1-i})\omega(v_{n-1-i}).
\end{equation}
Notice that the voltage of a dart is an element of $G$ whereas the voltage of a walk is a subset of $G$, even if such walk consists of a single dart. 

Let $\Gamma=\GC(\Delta,G,\omega,\zeta)$ and let $\varphi \colon \Gamma \to \Delta$ be the associated  generalised covering projection. If $(x_0, x_1, \ldots, x_{n-1})$ is a walk in $\Gamma$, then clearly its {\em projection} $(\varphi(x_0), \varphi(x_1), \ldots, \varphi(x_{n-1}))$ is
a walk in $\Delta$.

\begin{lemma}
\label{lem:endvertex}
Let  $(\Delta,G,\omega,\zeta)$ be a voltage graph with generalized cover $\Gamma$ and let $W$ be a $uv$-walk of length $n$ in $\Delta$. If $\lin{\mathcal{W}}$ is the set of walks of length $n$ in $\Gamma$ that start at the vertex $(u,\omega(u))$ and project to $W$, then the set of final vertices of walks in $\lin{\mathcal{W}}$ is precisely $\{(v,\omega(v)z): z \in \zeta(W)\}$.  
\end{lemma}

\begin{proof}
We proceed by induction on the length of $W$. If $W$ is a walk of length $1$, then the result follows at once from Lemma \ref{lem:neighbours}. Suppose Lemma \ref{lem:endvertex} holds for all walks of length at most $n$, and let $W=(x_0,x_1,...,x_n)$. Let $v_i = \beg x_i$ for all $i \in \{0,...,n\}$ and let $v_{n+1} = \term x_n$. Denote by $\lin{\mathcal{W}}$ the set of walks of length $n+1$ in $\Gamma$ that project to $W$. Let $W'=(x_0,x_1,...,x_{n-1})$ and let $\lin{\mathcal{W'}}$ the set of walks of length $n$ in $\Gamma$ that project to $W'$. Define $F$ and $F'$ as the set of final vertices of $\lin{\mathcal{W}}$ and $\lin{\mathcal{W'}}$, respectively. Then
$$F = \{\term (x,\omega(x)h) \mid \beg (x,\omega(x)h) \in F'\}.$$
But $F' = \{(v_n,\omega(v_n)z) \mid z \in \zeta(W')\}$ since $W'$ has length $n$. Then $\beg (x,\omega(x)h) \in F'$ if and only if $h \in \omega(v_n)\zeta(W')$. Thus
\begin{eqnarray*}
F &=& \{\term (x,\omega(x)h) \mid h \in \omega(v_n)\zeta(W')\}\\
&=& \{(v_{n+1},\omega(v_{n+1})\zeta(x_n)h) \mid h \in \omega(v_n)\zeta(W')\}\\
&=& \{(v_{n+1},\omega(v_{n+1})z) \mid z \in \zeta(x_n)\omega(v_n)\zeta(W')\}\\
&=& \{(v_{n+1},\omega(v_{n+1})z) \mid z \in \zeta(W)\}.
\end{eqnarray*}
The result follows.
\end{proof}

\begin{lemma}
\label{lem:gen}
Let $(\Delta,G,\omega,\zeta)$ be a voltage graph where $\zeta$ is $T$-normalized for some spanning tree $T$ of $\Delta$. Let $\mathcal{W}$ be the set of closed walks in $\Delta$ based at some vertex $v_0 \in \V(\Delta)$. 
Let $A= \langle \zeta(x) \mid x \in \D(\Delta) \setminus \D(T) \rangle$ and let
$B= \langle \omega(v) \mid  v \in \V(\Delta) \rangle$. 
Then $\langle \zeta(W) \mid W \in \mathcal{W}\rangle = \langle A, B \rangle$. 
\end{lemma}

\begin{proof}
First, note that by Equality~(\ref{eq:zetaW}), it follows that $\zeta(W)\in \langle A, B \rangle$ for all $W \in \mathcal{W}$. Hence $\langle \zeta(W) \mid W \in \mathcal{W}\rangle \leq \langle A, B \rangle$. 

We next show that $B \le \langle \zeta(W) \mid W \in \mathcal{W}\rangle$.
Let $v \in \V(\Delta)$ and denote by $P$ the unique $v_0v$-path in $T$. Further, let $\lin{P} = 
(x_0,x_1,\cdots,x_{n-1})$ be the closed walk $P \cdot P^{-1}$ based at $v_0$, obtained by concatenating the path $P$ and its inverse path $P^{-1}$. Note that $v=\beg(x_m)$ for $m=n/2$.
  Since every dart lying on $T$ has trivial voltage, $\zeta(\lin{P})= \prod_{i=0}^{n-1} \omega(\beg x_{n-1-i})$ and so $\omega(v) \subseteq \zeta(\lin{P}) \subseteq \{\zeta(W) \mid W \in \mathcal{W}\}$. This shows that $B \le \langle \zeta(W) \mid W \in \mathcal{W}\rangle$, as claimed.
  
  Now, let $x \in \D(\Delta)$. If $x$ lies on $T$, then $\zeta(x)=1_G \in \{\zeta(W) \mid W \in \mathcal{W}\}$. Suppose now $x\not\in \D(T)$. Then, there is a unique shortest closed walk $F_x$ based at $v_0$ such that $x$ is the only dart in $F_x$ not contained in $T$. 
  Since every dart of $F_x$, except possibly $x$, has trivial voltage, Equation~(\ref{eq:zetaW}) implies that $\zeta(x) \in \zeta(F_x)$. Since $F_x$ is a closed walk based at $v_0$, it follows that
$\zeta(x) \in \{\zeta(W) \mid W \in \mathcal{W}\}$, and thus $A\le \langle \zeta(W) \mid W \in \mathcal{W}\rangle$.
\end{proof}

\begin{theorem}
\label{lem:connect}
Let $(\Delta,G,\omega,\zeta)$ be a generalised voltage graph where $\zeta$ is $T$-normalized for some spanning tree $T$ of $\Delta$. Let $A= \langle \zeta(x) \mid x \in \D(\Delta) \setminus \D(T) \rangle$ and let
$B= \langle \omega(v) \mid  v \in \V(\Delta) \rangle$. Then $\GC(\Delta,G,\omega,\zeta)$ is connected if and only if $G = \langle A, B \rangle$. 
\end{theorem}

\begin{proof}
Let $\Gamma = \GC(\Delta,G,\omega,\zeta)$ and let $u,v \in \V(\Delta)$. Since $\Delta$ is connected, for every $\tilde{u} \in \fib(u)$, there exists a $\tilde{u}\tilde{v}$-walk in $\Gamma$ for some $\tilde{v} \in \fib(v)$; in other words, $\fib(v)$ can be reached from any vertex of $\Gamma$. Then, $\Gamma$ is connected if and only if for every vertex $(v,\omega(v)h) \in \fib(v)$, there is a walk connecting  $(v,\omega(v))$ to $(v,\omega(v)h)$. We will show that the latter happens if and only if $\langle A, B \rangle=G$.

Suppose $(v,\omega(v))$ can reach any other vertex in $\fib(v)$ by means of a walk. Such a walk projects in $\Delta$ to a closed walk based at $v$. Let $\mathcal{W}$ be the set of closed walks in $\Delta$ based at $v$ and let $\lin{\mathcal{W}}$ be the set of all walks of $\Gamma$ that project to a walk in $\mathcal{W}$. Define $\zeta_{\mathcal{W}}:=\{x\in\zeta(W)|W\in\mathcal{W}\}$. By Lemma \ref{lem:endvertex}, the end vertex of any walk in $\lin{\mathcal{W}}$ is $(v,\omega(v)g)$ for some $g \in \zeta_{\mathcal{W}}$. Then $\zeta_{\mathcal{W}}$ intersects every coset of $\omega(v)$, and since $\omega(v) \subseteq \zeta_{\mathcal{W}}$ we see that $G = \langle \zeta_{\mathcal{W}} \rangle = \langle A, B \rangle$, where the second equality follows from Lemma \ref{lem:gen}.

For the reverse implication, suppose $G = \langle \zeta_{\mathcal{W}} \rangle = \langle A, B \rangle$. Consider a vertex $(v,\omega(v)g) \in \fib(v)$. Then $g = g_{1}g_{2}\ldots g_{k}$ where $g_i \in \zeta_{\mathcal{W}}$ and thus $g_i \in \zeta(W_i)$ for some walk $W_i \in \mathcal{W}$, for all $i \in \{1,\ldots,k\}$. Clearly the concatenation $W_kW_{k-1}\ldots W_1$ is a walk in $\mathcal{W}$. Furthermore 
$$g = g_{1}g_{2}\ldots g_{k} \in \zeta(W_1)\zeta(W_2)\ldots\zeta(W_k) = \zeta(W_kW_{k-1}\ldots W_1).$$ 
Therefore, there is a walk $W'$ in $\Gamma$ connecting $(v,\omega(v))$ to $(v,\omega(v)g)$ that projects to  $W_kW_{k-1}\ldots W_1$.
\end{proof}

We conclude this section with an interesting application of the above theorem,  whose 
special cases, when $G$ acts transitively on the arcs of $\Gamma$ or on the edges but not on the arcs, are well-known and often-used facts in the algebraic graph theory. 
	
Let $\Gamma$ be a graph, let $G \in \Aut(\Gamma)$, let $\mathcal{T}$ be a transversal of the action of $G$ on $\V(\Gamma)\cup\D(\Gamma)$, let  
$\mathcal{T}_V = \mathcal{T} \cap \V(\Gamma)$,  let $\mathcal{T}_D = \mathcal{T} \cap \D(\Gamma)$, and let  $\mathcal{T}_D^\circ = \{x \in \mathcal{T}_D : x^{-1} \in \mathcal{T}_D\}$.
Further, let $\beg\vert_{\mathcal{T}_D}$ and $\inv\vert_{\mathcal{T}_D^\circ}$ be the restrictions
of the functions $\beg_\Gamma$ and $\inv_\Gamma$ to
 $\mathcal{T}_D$ and $\mathcal{T}_D^\circ$, respectively.
We say that $\mathcal{T}$ is {\em a connected transversal of $G$ on $\Gamma$} provided that the following holds:
(1) $\beg_\Gamma x \in \mathcal{T}_V$ for every $x\in \mathcal{T}_D$; 
(2) the graph $\Gamma[\mathcal{T}]:=(\mathcal{T}_V,\mathcal{T}_D^\circ, \beg\vert_{\D(\Gamma)}, \inv\vert_{\D(\Gamma)})$ is connected.

\begin{lemma}
\label{lem:contrans}
Let  $\Gamma$ be a graph and let $G\le \Aut(\Gamma)$. Then
 there exists a connected transversal of $G$ on $\Gamma$ if and only if
 the quotient graph $\Gamma/G$ is connected.
\end{lemma}

\begin{proof}
Let $\Delta:=\Gamma/G$. If $\Delta$ is connected, then it contains a spanning tree $T$.
By Theorem~\ref{the:lift2}, $\Gamma$ is isomorphic to $\GC(\Delta,G,\omega,\zeta)$
for some weight function $\omega$ and generalised voltage assignment $\zeta$.
Moreover, the isomorphism between $\Gamma$ and $\Gamma':=\GC(\Delta,G,\omega,\zeta)$
can be chosen to map $G$-orbits on $\Gamma$ to fibres of $\Gamma'$.
Furthermore, by Theorem~\ref{lem:normalized}, $\Gamma'$ is isomorphic to
$\Gamma'':=\GC(\Delta,G,\omega',\zeta')$, where $\zeta'$ is $T$-normalised.
Again, in view of Remark~\ref{rem:fib}, this isomorphism can be chosen to map fibres to fibres,
and hence there is an isomorphism $\Phi$ between $\Gamma$ and $\Gamma''$ which
maps $G$-orbits on $\Gamma$ to fibres of $\Gamma''$.

Let $X_V=\{(v,\omega'(v)) \mid v \in \V(T)\}$ and $X_D=\{(x,\omega'(x))\mid x \in \D(T)\}$. Let $\beg\vert_X$ and $\inv\vert_X$ be respective restrictions of $\beg_{\Gamma''}$ and $\inv_{\Gamma''}$ to $X_D$. Note that, since $\zeta'$ is $T$-normalised, $(x,\omega'(x)) \in X_D$ implies $(x,\omega'(x))^{-1} \in X_D$. It follows that $X:=(X_V,X_D,\beg\vert_X,\inv\vert_X)$ is a subgraph of $\Gamma''$. Moreover, $X$ is isomorphic to $T$. Now, since no two elements of $X_V \cup X_D$ belong to the same fibre of $\Gamma''$, $X$ can be extended to a set $\mathcal{T}$ containing exactly one element of each fibre of $\Gamma''$. Finally, since $\Phi$ maps $G$-orbits on $\Gamma$ to fibres of $\Gamma''$, we see that $\Phi^{-1} (\mathcal{T})$ is a connected transversal of $G$ on its action on $\Gamma$.

For the converse, observe that if $\mathcal{T}$ is a connected transversal, then its image under the quotient map is a connected spanning subgraph of $\Gamma /G$.

\end{proof}

\begin{theorem}
\label{the:genGv}
Let $\Gamma$ be a  graph, let $G\le \Aut(\Gamma)$ and let $\mathcal{T}$ be a connected transversal of $G$ on $\Gamma$.
Let $\mathcal{T}_V$, $\mathcal{T}_D$ and $\mathcal{T}_D^\circ$ be as in
the  preceding paragraph.
For a dart $x\in \mathcal{T}_D\setminus \mathcal{T}_D^\circ$ let
  $a_x\in G$ be such that $(x^{-1})^{a_x} \in \mathcal{T}_D$.
Then $\Gamma$ is connected if and only if 
$$G=\langle \bigcup_{v \in \mathcal{T}_V}  G_v\, \cup\, 
\{a_x : x \in \mathcal{T}_D\setminus \mathcal{T}_D^\circ \}\rangle.$$
\end{theorem}

\begin{proof}
Let $\Delta:=\Gamma/G$ and define $\zeta: \D(\Delta) \to G$ by 
$$
\zeta(x^G) = \left\{
        \begin{array}{ll}        	
            1_G & \hbox{ if } x \in \mathcal{T}_D^\circ;\\
            (a_x)^{-1} &  \hbox{ if } x \in \mathcal{T}_D\setminus \mathcal{T}_D^\circ.
        \end{array}
    \right.
$$ 
As in Thoerem~\ref{the:lift2}, for $x\in\mathcal{T}_D$ 
let $\iota(x)$ be the unique element of $\mathcal{T}_D$ such that $x^{-1} \in \iota(x)^G$,
and let
 $\omega: \V(\Delta) \cup \D(\Delta) \to S(G)$ be defined by
 $$\omega(x^G):=G_x \>\> \hbox{ for every }\>\> x \in \mathcal{T}.$$

We will first show that 
 $\zeta$ satisfies the condition (\ref{eq:zetacondition}) stated in Theorem~\ref{the:lift2}; that is, we will show that $\iota(x)^{\zeta(x^G)}=x^{-1}$ for every $x\in\mathcal{T}_D$.
 First suppose that $x \in \mathcal{T}_D\setminus \mathcal{T}_D^\circ$. Then $\iota(x)^{\zeta(x^G)} = \iota(x)^{(a_x)^{-1}}$ which by the definition of $a_x$ equals $x^{-1}$. On the other hand, if $x \in \mathcal{T}_D^\circ$, then $\iota(x) = x^{-1}$ and hence $\iota(x)^{\zeta(x^G)} = x^{-1}$. Therefore $\zeta$ satisfies condition the (\ref{eq:zetacondition})  stated in Theorem~\ref{the:lift2}.
  By Theorem~\ref{the:lift2} it follows that $(\Delta,G,\omega,\zeta)$ is a generalised
 voltage graph and $\Gamma\cong \GC(\Delta,G,\omega,\zeta)$.

Let $\Gamma[\mathcal{T}]$ be as in the paragraph preceding Theorem~\ref{the:genGv}. Since $\V(\Gamma[\mathcal{T}]) = \mathcal{T}_V$ is a transversal of the action of $G$ on $\V(\Gamma)$, it follows that the image $\Delta'$ of $\Gamma[\mathcal{T}]$
 under the quotient projection $\Gamma \to \Delta$ is a spanning subgraph of $\Delta$. Moreover,
since $\Gamma[\mathcal{T}]$ is connected, $\Delta'$ is connected and hence it contains
a spanning tree $T$ of $\Delta$. However, $\zeta$ is trivial on all the darts of $\Delta'$, implying that
$\zeta$ is $T$-normalised.

Let $A= \langle \zeta(X) \mid X \in \D(\Delta) \setminus \D(T) \rangle$ and $B= \langle \omega(v) \mid  v \in \V(\Delta) \rangle$, as in Theorem~\ref{lem:connect}. Observe that
$A= \langle \zeta(x^G) \mid x \in  \mathcal{T}_D\setminus \mathcal{T}_D^\circ \rangle =
\langle (a_x)^{-1} \mid x \in  \mathcal{T}_D\setminus \mathcal{T}_D^\circ \rangle$.
Similarly, $B=\langle G_v \mid v \in \mathcal{T}_V\rangle$.
 The result now follows from Theorem~\ref{lem:connect}.
\end{proof}

Note that Proposition \ref{prop:con} is the special case of Theorem~\ref{the:genGv} when the quotient graph $\Gamma/G$ is a tree.

\section{Simplicity}
\label{sec:simp}

In this section we  explore the simplicity of generalised covers. Lemmas \ref{lem:simple1} and \ref{lem:simple3} give necessary and sufficient conditions on a generalised voltage graph for its generalised cover to have, respectively, a pair of parallel darts or a self-inverse dart.
Since existence of a pair of parallel darts is equivalent to existence of a pair of parallel edges or a loop, and existence of a self-inverse dart is equivalent to existence of a semi-edge, these two lemmas yield a characterisation of simple generalised covers, given in  Theorem \ref{theo:simple}.

\begin{lemma}
\label{lem:simple1}
Let $(\Delta,G,\omega,\zeta)$ be a generalised voltage graph and let $\Gamma=\GC(\Delta,G,\omega,\zeta)$. Then there is a pair of parallel darts in $\Gamma$ if and only if $\zeta(y)h\zeta(x)^{-1} \in \omega(\term_\Delta(x))$ for some $x,y\in \D(\Delta)$ and
$h \in \omega(\beg_\Delta(x))$ such that $\beg_\Delta(x) = \beg_\Delta(y)$, $\term_\Delta(x) = \term_\Delta(y)$, and  either $x\not = y$, or $x=y$ but $h\not\in \omega(x)$.
\end{lemma}

\begin{proof}
For simplicity, we shall use the abbreviations $\beg$ and $\term$ for $\beg_\Delta$ and $\term_\Delta$, respectively, throughout this proof.

Suppose  there is a pair of parallel darts in $\Gamma$. Since $\Aut(\Gamma)$ acts
transitively on each dart-fibre, we may assume without loss of generality that
one of these two darts is of the form 
$(x,\omega(x))$ for some $x\in \D(\Delta)$. Then the other dart in the pair is of the form 
$(y,\omega(y)h)$ where
$$
\beg_\Gamma(x,\omega(x)) = \beg_\Gamma(y,\omega(y)h)\> \hbox{ and } \>
\term_\Gamma(x,\omega(x)) = \term_\Gamma(y,\omega(y)h).
$$
The first equality implies that 
$\beg(x) = \beg(y)$ and 
$\omega(\beg(x)) = \omega(\beg(y))h$.
Similarly, from the second equality, we deduce that
$\term(x) = \term(y)$ and
$\omega(\term(x))\zeta(x) = \omega(\term(y))\zeta(y)h$.
In particular, $\zeta(y)h\zeta(x)^{-1} \in  \omega(\term(x))$.
Since $(x,\omega(x)) \not= (y,\omega(y)h)$, we also see that
either $x\not = y$, or $x=y$ but $h\not\in \omega(x)$.

Conversely, if $\zeta(y)h\zeta(x)^{-1} \in \omega(\term(x))$ for some $x,y\in \D(\Delta)$ and
$h \in \omega(\beg(x))$ such that $\beg(x) = \beg(y)$, $\term(x) = \term(y)$, 
and  either $x\not = y$, or $x=y$ but $h\not\in \omega(x)$, 
then we see that the darts $(x,\omega(x))$ and $(y,\omega(y)h)$ are distinct. 
Furthermore, $\omega(\term(x))\zeta(x) = \omega(\term(x))\zeta(y)h = \omega(\term(y))\zeta(y)h$, from which we see that $\term_\Gamma(x,\omega(x))=\term_\Gamma(y,\omega(y)h)$. Finally, since $h \in \omega(\beg(x))=\omega(\beg(y))$ we see that $\omega(\beg(x)) = \omega(\beg(y))h$ and so $\beg_\Gamma(x,\omega(x))=\beg_\Gamma(y,\omega(y)h)$. Therefore $(x,\omega(x))$ and $(y,\omega(y)h)$ are parallel darts of $\Gamma$.
\end{proof}

\begin{lemma}
\label{lem:simple3}
Let $(\Delta,G,\omega,\zeta)$ be a generalised voltage graph and let 
$\Gamma=\GC(\Delta,G,\omega,\zeta)$. 
Then $\Gamma$ has a semi-edge if and only if
there is $x\in \D(\Delta)$ such that $\zeta(x) \in \omega(x)$.
%
\end{lemma}

\begin{proof}
Recall that a graph has a semi-edge if and only if it has a self-inverse dart. 
Since $\Aut(\Gamma)$ is transitive on each fibre, we may assume that if such a dart exists,
then it is of the form $(x,\omega(x))$ for some $x\in \D(\Delta)$. Recall that the inverse of this dart is
$(x^{-1},\omega(x^{-1})\zeta(x))$. Hence the dart $(x,\omega(x))$
is self-inverse if and only if $x=x^{-1}$ and $\zeta(x) \in \omega(x)$.
\end{proof}

The following characterisation of generalised voltage graph that yield simple generalised covers
 is a direct consequence of Lemmas~\ref{lem:simple1} and \ref{lem:simple3}.

\begin{theorem}
\label{theo:simple}
Let $(\Delta,G,\omega,\zeta)$ be a generalised voltage graph and let $\Gamma=\GC(\Delta,G,\omega,\zeta)$. Then $\Gamma$ is a simple graph if and only if for all darts $x\in \D(\Delta)$
all of the following conditions hold:
\begin{enumerate}
\item $\zeta(x)h\zeta(x)^{-1} \notin \omega(\term_\Delta(x))$ for all $h \in \omega(\beg_\Delta(x)) \setminus \omega(x)$;         
\item  $\zeta(y)h\zeta(x)^{-1} \not\in \omega(\term_\Delta(x))$
          for all darts $y$, $y\not = x$, which are parallel to $x$ in $\Delta$
          and for all $h\in \omega(\beg_\Delta(x))$;
\item $\zeta(x)\not \in \omega(x)$ if $x=x^{-1}$.
\end{enumerate}
\end{theorem}

\end{document}